\documentclass[11pt, a4paper]{amsart}
\usepackage{amsmath, amsthm, amssymb}

\usepackage{mathtools}
\mathtoolsset{showonlyrefs}

\newcommand\blfootnote[1]{%
\begingroup
\renewcommand\thefootnote{}\footnote{#1}%
\addtocounter{footnote}{-1}%
\endgroup
}

\newtheorem{theorem}{Theorem}[section]
\newtheorem{corollary}[theorem]{Corollary}
\newtheorem{lemma}[theorem]{Lemma}
\newtheorem{definition}[theorem]{Definition}
\newtheorem{problem}{Problem}
\newtheorem*{claim*}{Claim}

\newenvironment{potwr}[2][Theorem]{{\noindent\it \underline{Proof of #1 \ref{#2}}}:\\}{\qed}

\def \vol {\operatorname{vol}}
\def \domlambda {(\frac n{n+p},1) \cup (1, +\infty]}
\def \R {\mathbb R}
\def \l {\lambda}
\def \lp {\lambda'}

\def \co {\operatorname{co}}
\def \gl {\operatorname{Gl}_n(\R)}
\def \ball {\mathbb B_n}
\def \supp {\operatorname S}
\def \as {\operatorname{as}}

\newcommand{\many}[2][n]{{#2}_1, \ldots, {#2}_{#1}}
\newcommand{\mani}[1]{\many[n-1]{#1}}
\newcommand{\man}[1]{{#1}, \ldots, {#1}}
\newcommand{\manyintn}[2][n]{\int_{#2_1} \cdots \int_{{#2}_{#1}}}
\newcommand{\maniintn}[1]{\manyintn[n-1]{#1}}
\newcommand{\manyint}[1]{\int_{#1} \cdots \int_{#1}}
\newcommand{\manydx}[1][x]{d#1_1 \cdots d#1_n}
\newcommand{\manidx}[1][x]{d#1_1 \cdots d#1_{n-1}}

\def \constrsiu {b_n}
\def \constrsis {b_{n,p}}
\def \constrsif {B_{n,p,\lambda}}
\def \constrsids {\tilde b_{n,p}}
\newcommand \constrsidf[1][\alpha] {\tilde B_{n,p,#1}}
\def \constisos {a_{n,p}}
\def \constisof {A_{n,p,\lambda}}

\def \constpettyconjval {\left(\frac{\omega_{n-1}}{\omega_n}\right)^n \omega_n^2}
\def \constcentroidu {c_{n,1}}
\def \constcentroidp {c_{n,p}}
\def \constcnv { \operatorname{cnv}_{n,p} }
\def \constsob { \mathcal S_{n,p} }
\def \constlevelsets {L_{n,p,\lambda}}
\def \constzhang {c_{n,p}}
\newcommand \constmoment[1][\l] {m_{n,p,#1}}
\def \extrmoment {G} 
\def \extrsobolev {F} 

\def \constconjrsids {\bar b_{n,p}}
\newcommand \constconjrsidf[1][p] {\bar B_{n,#1}}
\def \constconjisods {\bar a_{n,p}}
\def \constconjisodf {\bar A_{n,p}}
\def \I {I_p}%
\def \N {N_p}%
\def \Nc {N_p^\circ}%
\def \V {V_p}%
\newcommand \tI[1][p] {\tilde I_{#1}}
\def \tN {\tilde N_p}
\def \tV {\tilde V_{-p}}
\def \brsi {Busemann Random Simplex Inequality\xspace}
\usepackage{xspace}
\def \cb {convex body\xspace}
\def \sb {$C^2$-smooth boundary\xspace}
\def \pc {positive Gauss curvature\xspace}
\newcommand \cf[1][s] {continuous non-negative function#1 with compact support\xspace}
\newcommand \cfs[1][2] {$C^{#1}$ smooth, non-negative functions with compact support\xspace}
\newcommand \eq {equality holds if and only if\xspace}

\newcommand \equa[1] {equality holds if and only if $#1$ is an origin-symmetric ellipsoid\xspace}
\newcommand \equas[1] {equality holds if and only if $#1$ are homothetic origin-symmetric ellipsoids\xspace}

\newcommand \equasi[2] {equality holds in \eqref{#1} if and only if $#2$ are homothetic origin-symmetric ellipsoids\xspace}
\date{}

\title[A Convex Body Associated to the Busemann Random Simplex Inequality]{A Convex Body Associated to the Busemann Random Simplex Inequality and the Petty conjecture}
\author{J. Haddad}
\address{Juli\'an Haddad: Departamento de Matem\'atica, ICEx,  Universidade Federal de Minas Gerais, 30.123-970, Belo Horizonte, Brasil.}
\email{jhaddad@mat.ufmg.br}

\begin{document}
\blfootnote{\textup{2000} \textit{Mathematics Subject Classification}: \textup{Primary 52A22, Secondary 46E35}}


\maketitle

\begin{abstract}
	Given $L$ a \cb, the $L_p$-\brsi is closely related to the centroid body $\Gamma_p L$ for $p=1$ and $2$, and only in these cases it can be proved using the $L_p$-Busemann-Petty centroid inequality. We define a \cb $\N L$ and prove an isoperimetric inequality for $(\N L)^\circ$ that is equivalent to the $L_p$-\brsi.
	As applications, we give a simple proof of a general functional version of the \brsi and study a dual theory related to Petty's conjectured inequality. More precisely, we prove dual versions of the $L_p$-\brsi for sets and functions by means of the $p$-affine surface area measure, and we prove that the Petty conjecture is equivalent to an $L_1$-Sharp Affine Sobolev-type inequality that is stronger than (and directly implies) the Sobolev-Zhang inequality.
\end{abstract}

\section{Introduction}
Let $L \subseteq \R^n$ be a \cb (a compact convex set with non-empty interior) with the origin as interior point. In 1953 Busemann showed \cite{busemann1953volume} that if $n$ points are chosen randomly inside $L$ with uniform probability, then the expected volume of the simplex formed by the convex hull of these points and the origin attains its minimum among all convex sets of the same volume, when $L$ is an ellipsoid. 
More precisely, if we denote
\[D_n(\many{v}) = |\det(\many{v})|,\]
then $D_n(\many v)$ is the volume of the parallelepiped spanned by the vectors $\many v$.
We have \\$\frac 1{n!}D_n(\many v) = \vol(\co(0,\many v))$ and the result can be stated as
\begin{equation}
	\label{ineq_rsi_1s}
	\manyint{L} D_n(x_1, \ldots, x_{n}) \manydx \geq \constrsiu \vol(L)^{n+1}
\end{equation}
where $\constrsiu$ is a sharp constant, and \equa{L}.
This is known as the \brsi.

A \cb may be characterized by its support function $h_K$, defined as
\[
h_K(y)=\max\{\langle y, z\rangle :\ z\in K\}\, .
\]
It describes the (signed) distance of supporting hyperplanes of $K$ to the origin.
Closely related to inequality \eqref{ineq_rsi_1s}, the centroid body of $L$ is defined as the unique \cb $\Gamma_1 L$ having as support function 
\[h(\Gamma_1 L, \xi) = \frac 1{\constcentroidu \vol(L)} \int_L |\langle \xi, x \rangle| dx\]
where the constant $\constcentroidu$ is such that if $\ball$ is the unit euclidean ball, $\Gamma_1 \ball = \ball$.

It is not difficult to show (see \cite{schneider2014convex}, formulas 10.69 and 5.82) that
\[\vol(\Gamma_1 L) = \constcentroidu^{-n} \frac{2^n}{n!} \vol(L)^{-n} \manyint{L} D_n(x_1, \ldots, x_{n}) \manydx \]
and inequality \eqref{ineq_rsi_1s} is equivalent to the well known Busemann-Petty centroid inequality
\[\vol(\Gamma_1 L) \geq \vol(L),\]
where \equa{L}.

Several generalizations and extensions of inequality \eqref{ineq_rsi_1s} have been studied.
Groemer proved in \cite{groemer1973some} that the expected value, as well as the higher order moments of the volume of the convex hull of $n+1$ points inside $K$ is minimized when $K$ is an ellipsoid. Then in \cite{groemer1974mean} extended the result to the case where the number of points $k$ is allowed to be greater than or equal to $n+1$.
Pfiefer \cite{pfiefer1990maximum} extended the result to measurable sets, and were the volume is composed with an increasing function.
Hartzoulaki and Paouris \cite{hartzoulaki2003quermassintegrals} proved it replacing the volume by the Quermassintegrals.
The reverse inequality, with triangles as equality cases (in the planar case) was studied by Campi, Colesanti and Gronchi \cite{campi1999note, campi2006extremal} and Saroglou \cite{saroglou2010characterizations}.
Dann, Paouris and Pivovarov in \cite{dann2016bounding} showed functional inequalities for any number of points $k \geq 1$, and Paouris and Pivovarov in \cite{paouris2017randomized} showed a general condition allowing to apply Steiner symmetrization to many isoperimetric inequalities in stochastic form.

A general version of inequality \eqref{ineq_rsi_1s} where the $k$ bodies in the multiple integral are allowed to be different, and the volume of the parallelepiped is raised to the power $p>0$, was proved in \cite{bourgain1988geometric}. We state here a particular case.
\begin{theorem}
	\label{thm_rsi_s}
Let $\many L$ be convex bodies and $p\geq 1$, define
	\[\I(\many L) = \manyintn L D_n(\many x)^p \manydx\]
then we have
\begin{equation}
\label{ineq_rsi_s}
	\I(\many L) \geq \constrsis \prod_{i=1}^n \vol(L_i)^{\frac{n+p}n}
\end{equation}
\end{theorem}
where $\constrsis$ is such that \equas{\many L}.
We refer to \cite[Theorem 8.6.1]{schneider2008stochastic} for the explicit value of $\constrsis$.

It is important to mention that the proofs of these theorems are all based on some form of Steiner symmetrization.
We propose a different approach based on the construction of a \cb $\N (\mani L)$.
In \cite{lutwak1997blaschke,lutwak2000lp} Lutwak, Yang and Zhang defined the $L_p$ Centroid body $\Gamma_p L$ and proved an $L_p$ version of the Busemann-Petty centroid inequality
\begin{equation}
	\label{ineq_bp_p}
	\vol(\Gamma_pL) \geq \vol(L),
\end{equation}
where \equa L.
The body $\Gamma_p L$ is defined by the support function
\[h(\Gamma_pL, \xi)^p = \frac 1{\constcentroidp \vol(L)} \int_L |\langle \xi, x \rangle|^p dx\]
where $\constcentroidp$ is such that $\Gamma_p \ball = \ball$.
In general, it doesn't seem to exist any relation between inequalities \eqref{ineq_bp_p} and \eqref{ineq_rsi_s}, except for $p=1$ and $p=2$, (see \cite[p.~4]{campi2006extremal}).
The main purpose of this paper is to construct a \cb containing the information of the $L_p$-\brsi, playing the role of the centroid body for $p \geq 1$, and obtain some inequalities related to it.
The inequalities we obtain are already known to be true in some particular cases, but the method is new and very simple, no symmetrization arguments are needed.
The second objective is to study a dualization of Theorem \ref{thm_rsi_s} that is suggested by the proof of Theorem \ref{thm_iso_f} below.

We start with the (fairly obvious) definition.
\begin{definition}
	Let $\mani L$ be convex bodies, we define $\N (\mani L)$ as the \cb defined by the support function
	\[
		h(\N (\mani L), \xi)^p = \maniintn L D_n(\mani x, \xi)^p \manidx.
	\]
	We write $\N L = \N(\man L)$.
\end{definition}

The polar of a \cb $K$ is the \cb defined by 
\[K^\circ = \{x \in \R^n \ /\ h(K,x) \leq 1\}.\]
We shall prove the following
\begin{theorem}
	\label{thm_iso_s}
	Let $\mani L$ be convex bodies, then
\begin{equation}
	\label{ineq_iso_s}
	\vol(\Nc (\mani L)) \leq \constisos \prod_{i=1}^{n-1} \vol(L_i)^{-\frac{n+p}p},
\end{equation}
	where $\constisos = \left( \frac {n+p}n \constrsis\right)^{-\frac np}$ and $\Nc (\mani L)$ stands for the polar body of $\N  (\mani L)$.
\end{theorem}
\noindent The inequality is invariant under volume-preserving affine transformations.
Also, \equasi{ineq_iso_s}{\mani L}.
Moreover, Theorem \ref{thm_iso_s} is equivalent to Theorem \ref{thm_rsi_s}, and the only tool we use to prove this equivalence is the dual mixed volume inequality.
Let us define for two convex bodies $K,L$, the dual mixed volume as
\[\tV (K,L) = \frac{n+p}n \int_K h_{L^\circ}(x)^p dx.\]
The dual mixed volume inequality states that $\tV (K,L) \geq \vol(K)^{\frac{n+p}n} \vol(L)^{-\frac pn}$.
Then it suffices to note that
\[\I(\many L) = \frac n{n+p} \tV(L_n, \Nc (\mani L)).\]
If we know Theorem \ref{thm_iso_s} to be true we obtain
\[\I(\many L) \geq \frac n{n+p} \vol(L_n)^{\frac {n+p}n} \vol(\Nc (\mani L))^{-\frac pn} \geq \constrsis \vol(L_n)^{\frac {n+p}n} \prod_{i=1}^{n-1} \vol(L_i)^{\frac{n+p}n}.\]
Conversely, if we know Theorem \ref{thm_rsi_s} to be true, we can compute
\begin{align*}
\vol(\Nc (\mani L)) 
	&= \tV(\Nc (\mani L), \Nc (\mani L))\\
	&= \frac {n+p}n \I(\mani L, \Nc (\mani L)) \\
	&\geq \frac {n+p}n \constrsis \vol(\N ^\circ (\mani L) )^{\frac{n+p}n} \prod_{i=1}^{n-1} \vol(L_i)^{\frac{n+p}n},
\end{align*}
and obtain Theorem \ref{thm_iso_s}.

The definition of $\N (\mani L)$ fits perfectly to prove a functional version of inequality \eqref{ineq_rsi_s}.
Let $\mani l:\R^n \to \R$ be \cf, and let us define
\[\I(\many l) = \manyint{\R^n} l_1(x_1) \cdots l_n(x_n) D_n(\many x)^p \manydx.\]
Notice that we recover the previous definition of $\I$ if we take the functions $l_i$ to be indicator functions of convex bodies.
We define accordingly the set $\N (\mani l)$.
\begin{definition}
	Let $\mani l$ be \cf, we define $\N (\mani l)$ as the \cb defined by the support function
	\[
		h(\N (\mani l), \xi)^p = \manyint{\R^n} l_1(x_1) \cdots l_{n-1}(x_{n-1}) D_n(\mani x, \xi)^p \manidx.
	\]
\end{definition}
The subject of functional inequalities with geometric counterpart attracted great interest in recent years (see for example \cite{fradelizi2007some, caglar2015functional, haddad2018asymmetric}), specially affine invariant functional inequalities, that often imply an euclidean version.
For this reason we propose to study several extensions of our results to the functional setting.

	Let us define the function
	\begin{align*}
		\extrmoment_{p,\l}(s) = \left\{
			\begin{array}{cc}
				(1+|s|^p)^{1/(\l-1)}, & \l < 1\\
				(1-|s|^p)_+^{1/(\l-1)}, & \l > 1\\
				\chi_{[-1,1]}(s), & \l = + \infty.\\
			\end{array}
			\right.\,
	\end{align*}
Here $t_+ = \max\{t,0\}$ and $\chi_{[-1,1]}$ denotes the characteristic function of the set $[-1,1]$.

\begin{theorem}
	\label{thm_iso_f}
Let $\mani{l}$ be \cf, and take any $\lambda \in \domlambda$. Then,
	\begin{equation}
		\label{ineq_iso_f}
		\vol(\Nc (l_1, \ldots,l_{n-1})) \leq \constisof \prod_{i=1}^{n-1} \|l_i\|_1^{-\frac{n+p\lp}p} \|l_i\|_\l^{\lp}
	\end{equation}
	where $\constisof$ is a sharp constant
	and \eq the functions $l_i$ have the form $l_i(x) = a_i \extrmoment_{p,\l}(|b_i A.(x-x_i)|)$ with $x_0 \in \R^n$, $a_i, b_i > 0$ and $A \in \gl$.
	The explicit value of $\constisof$ is computed in the proof of Theorem.
\end{theorem}

The equivalence between Theorems \ref{thm_rsi_s} and \ref{thm_iso_s} remains valid in the functional setting
and we obtain
\begin{corollary}
	\label{cor_rsi_f}
	Let $\many l$ be \cf, and take any $\lambda \in \domlambda$. Then,
	\begin{equation}
	\label{ineq_rsi_f}
		\I(\many l) \geq \constrsif \prod_{i=1}^n \|l_i\|_1^{\frac{n+p\lp}n} \|l_i\|_\l^{- \frac{p \lp}n }
	\end{equation}
	and \eq the functions $l_i$ are extremal functions of Theorem \ref{thm_iso_f}.
	The explicit value of $\constrsif$ is computed in the proof of the Corollary.
\end{corollary}
The case $\lambda = \infty$ of Corollary \ref{cor_rsi_f} appears already in \cite{paouris2017randomized}, application number 10, Section 5.
A more general version where the number of vertices in the random simplex is $k \leq n$ appears in \cite[Corollary~4.2]{dann2016bounding}.
Throughout this paper the operator $D_n$ can be replaced by any positive function being homogeneous of degree $1$ in each variable, and any number of variables.
Considering for example $D_k(\many[k] v)$ defined as the $k$ dimensional volume of the parallelepiped spanned by the vectors $\many[k] v$, since \cite[Corollary~4.2]{dann2016bounding} proves Theorem \ref{thm_rsi_s} for $k\leq n$ points, then the equivalence between Theorem \ref{thm_rsi_s} and Corollary \ref{cor_rsi_f} applies, and we recover a general version of \cite[Corollary~4.2]{dann2016bounding} for $\lambda \in \domlambda$.

The proof of Theorem \ref{thm_iso_f} uses a moment inequality (Lemma \ref{lem_functionaldualmixedvolume}) proved in \cite{lutwak2004moment} (this is a dual mixed volume inequality for functions), and an induction argument.
Given the simplicity of the proofs of Theorem \ref{thm_iso_f} and Corollary \ref{cor_rsi_f}, we propose to study a dual random process and an associated convex set $\tilde \N (\mani{L})$.
The motivation is the following: since all the tools involved in the proofs of Theorem \ref{thm_iso_f} and Corollary \ref{cor_rsi_f} have dual versions (the dual mixed volume and the dual mixed volume inequality), we can also prove an equivalence between a volume inequality for the set $\tilde \N (\mani{L})$ and an inequality for a dual functional $\tI(\many L)$ by means of the mixed volume and the mixed volume inequality.
This way we arrive to what we believe should be the correct dualization of $\I(\many L)$.

\begin{definition}
Let $\many{L}$ be convex bodies and $\many{l}$ be \cf, define 
\begin{align*}
	\tI(\many L) &= \manyint {S^{n-1}} D_n(\many{\xi})^p dS_{p, L_1}(\xi_1) \ldots dS_{p, L_n}(\xi_n)\\
	\tI(\many l) &= \manyint{\R^n} D_n(\nabla l_1(x_1), \ldots, \nabla l_n(x_n) )^p \manydx \\
\end{align*}
where $S_{p,L}$ denotes the $L_p$ surface area measure (see Section \ref{sec_notation}).
	Define the \cb $\tN (\mani L)$ by the support function
	\[
		h(\tN(\mani L), \xi)^p = \manyint{S^{n-1}} D_n(\mani \xi, \xi)^p dS_{p, L_1}(\xi_1) \ldots dS_{p, L_{n-1}}(\xi_{n-1})
	\]
\end{definition}
The quantity $S_{p,L}(L)^{-n} \tI(\man L)$ represents the expected $p$-th power of the volume of a random parallelepiped generated by the normal vectors $n_{x_1}, \cdots, n_{x_n}$ at points $x_i$ chosen randomly in the surface of $L$, with probability measure $S_{p,L}/S_{p,L}(L)$.
The definition of $\tI(\many l)$ is based on the definition of the surface area measure of a function (see Section \ref{sec_notation}), and for $p=1$ can be interpreted in the weak sense if $\many l$ are indicator functions of convex bodies.

The inequalities that should hold with respect to $\tI$ and $\tN$ in order to dualize the proof of Theorem \ref{thm_iso_f} are summarized in Section \ref{sec_open} but they are open problems.
Let us mention the following: for $1\leq p < n$ we ask if the inequality
\begin{equation}
	\label{conjineq_rsid_s_first}
	\tI(\many L) \geq \constconjrsids \prod_{i=1}^n \vol(L_i)^{\frac{n-p}n}
\end{equation}
holds, where the sharp constant is such that \equasi{conjineq_rsid_s_first}{\many L} for $p>1$ and homothetic ellipsoids for $p=1$.
Again, the case $p=1$ is special. The projection body of a \cb $K$ is defined by the support function
\[h(\Pi K, \xi) = \frac 12 \int_{S^{n-1}} |\langle \xi, \eta \rangle| dS_K(\eta)\]
and it is easy to see that $\vol(\Pi K) = \frac 1{n!} \tI[1](\man K)$.
Then inequality \eqref{conjineq_rsid_s_first} with $p=1$ and $L_1 = \cdots = L_n = L$ is Petty's conjectured inequality, one of the major open problems in Convex Geometry.
This is that
\begin{equation}
	\label{ineq_pettyconj}
	\vol(L)^{1-n} \vol(\Pi L) \geq \constpettyconjval
\end{equation}
where \equa L.
Also, as a consequence of the Aleksandrov-Fenchel inequality (see \cite[(1.7)]{lutwak1985mixed} or \cite[(7.64)]{schneider2014convex} for the proof), we have
\[\frac 1{n!} \tI[1](\many L) = V(\many{\Pi L}) \geq \prod_{i=1}^n \vol(\Pi L_i)^{1/n}\]
so \eqref{ineq_pettyconj} is also equivalent to \eqref{conjineq_rsid_s_first} with $p=1$ and arbitrary convex bodies $\many L$.
We refer to \cite[Section 10.9]{schneider2014convex} for an overview about recent developments and main difficulties around this conjecture.
In Section \ref{sec_open} we shall prove that \eqref{conjineq_rsid_s_first} for $p=1$ (and thus the Petty conjecture) implies the sharp, affine invariant Sobolev-like inequality
\begin{equation}
	\label{ineq_sobolevish_1_first}
	\|f\|_{\frac n{n-1}} \leq C_n \left( \manyint{\R^n} D_n(\nabla f(x_1), \ldots, \nabla f(x_n)) \manydx \right)^{\frac 1n}
\end{equation}
that is stronger than (and directly implies) the sharp affine Sobolev inequality proved by Zhang \cite{zhang1999affine}.

In Section \ref{sec_rsid} we prove a weaker result that is consequence of Corollary \ref{cor_rsi_f}.
\begin{theorem}
	\label{thm_rsid_s}
	Let $\many{L}$ be convex bodies with \sb and \pc, then
	\begin{equation}
		\label{ineq_rsid_s}
		\tI(\many L) \geq \constrsids \prod_{i=1}^n \Omega_p(L_i)^{\frac {n+p}n}
	\end{equation}
	where $\constrsids = (n+p)^n \frac \constrsis {n^{n+p}}$ and $\Omega_p(L)$ denotes the $p$-affine surface area of $L$
and \equas{\many L}.
\end{theorem}
Theorem \ref{thm_rsid_s} is weaker than inequality \eqref{conjineq_rsid_s_first} and it falls short for polytopes since $\Omega_p(L_i)=0$ in that case.
We believe Theorem \ref{thm_rsid_s} can be improved in this regard.
Notice that for $p=1$ and $L_1 = \cdots = L_n = L$ we recover Petty's affine projection inequality \cite{petty1965projection}, see also \cite[Section 10.9]{schneider2014convex}.

Finally, a functional version of Theorem \ref{thm_rsid_s} reads as follows.
\begin{theorem}
	\label{thm_rsid_f}
	Let $\many l$ be non-zero, \cfs.
	Assume additionally that $\many l$ have convex level sets with \pc.
	Let $\alpha \in (\frac n{n+1},1)\cup(1,\infty)$, then
	\begin{equation}
		\label{ineq_rsid_fc}
		\tI(\many l) \geq \constrsidf \prod_{i=1}^n  \left( \Omega_p(l_i)^{\frac{n+\alpha'}{n+1}} \Omega_p(l_i^\alpha)^{-\frac 1{n+1} \frac 1{\alpha-1}} \right)^{\frac{n+p}n}
	\end{equation}
	where
	\[\Omega_p(l) = \int_{\R^n} |\det(K l(x))|^{\frac p{n+p}} dx
	\hbox{,\ and \ }
	K l(x) = \left( \begin{array}{c|c} 0 & \nabla l(x)^T\\ \hline \\ \nabla l(x) & Hl(x) \end{array} \right)\]
	and with equality if each function $l_i$ has the form $l_i(x) = a_i F(b_i |A.x|_2)_+$ where $a_i, b_i > 0$, $A \in \gl$ and $F:(0,T] \to [0,1)$ is a solution of the following ODE
	\[
		F'(t) = \left\{
		\begin{array}{cc}
			-t^{-\frac{n-1}{p-1}} (1-F(t)^{\frac{(\alpha -1) (n+1) p}{n+p}})^{\frac{n+p}{p(p-1)}} &\hbox{ if } \alpha > 1\\
			 -t^{-\frac{n-1}{p-1}} (F(t)^{\frac{(\alpha -1) (n+1) p}{n+p}}-1)^{\frac{n+p}{p(p-1)}} &\hbox{ if } \alpha < 1\\
		\end{array}
		\right.
	\]
	with $F(T) = 0$.
	For $\alpha = \infty$ we obtain
	\[\tI(\many l) \geq \constrsidf[\infty] \prod_{i=1}^n   \Omega_p(l_i)^{\frac{n+p}n} \|l_i\|_\infty^{-\frac pn}.\]
The explicit value of $\constrsidf$ is computed in the proof of the Theorem.
\end{theorem}
The quantity $\Omega_p(l)$ at the right-hand side can be interpreted as a functional version of the $p$-affine surface area measure, although other definitions can be found in the literature that are slightly different.
For example, in formula (26) of \cite{caglar2015functional}, it is defined for a smooth convex function $\psi$ as
\[ \as_\lambda(\psi) = \int_{\R^n} \det(H \psi(x))^\lambda e^{-\psi(x)} d x. \]
It can be shown that if $l(x) = F(\|x\|_K)$ with $F:[0,\infty) \to \R_+$ a $C^2$ smooth (not necessarily decreasing) function with compact support, then $\Omega_p(l) = \Omega_p(K)$
provided that $\int_0^\infty t^{\frac{n(n-1)}{n+p}} |F'(t)|^{\frac {p(n+1)}{n+p}} dt = 1$.
We remark that in Theorem 3 of \cite{caglar2015functional}, a similar relation for $\as_\lambda(\psi)$ is proved only for $F(t) = \frac{t^2}2$.
For the case $p=1$ and any $l(x)$ with convex level sets, our definition also satisfies
\[
	\Omega_1(l) = \int_0^\infty \Omega_1(\{x \in \R^n/\ l(x) \geq t \} ) dt.
\]
If $A \in \gl$ and $l_A(x) = l(A.x)$ then 
\[K l_A(x) = \left(\begin{array}{c|c} 1 & 0\\ \hline \\0 & A^T \end{array}\right). K l(x) . \left(\begin{array}{c|c} 1 & 0\\ \hline \\0 & A \end{array}\right)\]
so both sides of inequality \eqref{ineq_rsid_fc} are invariant under volume preserving affine transformations.

The rest of the paper is organized as follows. In Section \ref{sec_notation} we fix some notations and present some background in convex geometry to be used throughout the paper.
In Section 3 we define the set $\N $ and prove Theorem \ref{thm_iso_f} and Corollary \ref{cor_rsi_f}.
In Section 4 we study the dual \brsi and prove Theorems \ref{thm_rsid_s} and \ref{thm_rsid_f}.
Finally in Section 5 we discuss the conjectured inequality \eqref{conjineq_rsid_s_first}, and its relation to the Petty conjecture.

This work was partly written during a short visit at TU Wien.
We would like to thank professor F. Schuster and his team of collaborators for their hospitality, useful comments and fruitful discussions.

The author was partially supported by Fapemig, Project APQ-01542-18.
                                   
\section{Notation and Preliminaries}
\label{sec_notation}
This section is devoted to basic definitions and notations within the convex geometry. For a comprehensive reference in convex geometry we refer to the book \cite{schneider2014convex}.

We recall that a \cb $K\subset\R^n$ is a convex compact subset of $\R^n$ with non-empty interior.

The support function $h_K$ is defined as
\[
h_K(y)=\max\{\langle y, z\rangle /\ z\in K\}, 
\]
and uniquely characterizes $K$. If $K$ contains the origin in the interior, then we also have the gauge $\|\cdot\|_K$ and radial $r_K(\cdot)$ functions of $K$ defined respectively as
\[
\|y\|_K:=\inf\{\lambda>0 :\  y\in \lambda K\}\, ,\quad y\in\R^n\setminus\{0\}\, ,
\]
\[
r_K(y):=\max\{\lambda>0 :\ \lambda y\in K\}\, ,\quad y\in\R^n\setminus\{0\}\, .
\]

Clearly, $\|y\|_K=\frac{1}{r_K(y)}$. We also recall that $\|\cdot\|_K$ is actually a norm when the \cb $K$ is centrally symmetric, i.e. $K=-K$, and the unit ball with respect to $\|\cdot\|_K$ is just $K$.
On the other hand, a general norm on $\R^n$ is uniquely determined by its unit ball, which is a centrally symmetric \cb.

For a \cb $K\subset \R^n$ containing the origin in its interior we define the polar body, denoted by $K^\circ$, by
\[
K^\circ:=\{y\in\R^n /\ \langle y,z \rangle\leq 1\quad \forall z\in K\}\, .
\]
Evidently, $h_K^{-1} = r_{K^{\circ}}$.

For $p \geq 1$, the $L_p$-mixed volume $\V(K,L)$ of convex bodies $K$ and $L$ is defined by
\begin{equation}
	\V(K,  L) = \frac{p}{n}\lim_{\varepsilon \rightarrow 0}\frac{\vol(K +_p \varepsilon\cdot_p L) - \vol(K)}{\varepsilon},
\end{equation}
where $K +_p \varepsilon \cdot_pL$ is the \cb defined by:
\begin{equation}
	h_{K +_p \varepsilon \cdot_p L}(x)^p = h_K(x)^p + \varepsilon h_L(x)^p, \quad \forall x \in \mathbb{R}^n.
\end{equation}

It is known (see \cite{LEBMFtheory}) that there exists a unique finite positive Borel measure $S_p(K, .)$ on $\mathbb{S}^{n-1}$ such that
\begin{equation}
\V(K, L) = \frac{1}{n}\int_{\mathbb{S}^{n-1}}h_L(u)^pdS_p(K, u),
\end{equation}
for each \cb $L$.

If $1 \leq p < \infty$ and $K, L$ are convex bodies in $\mathbb{R}^n$ containing the origin as interior point, we can find also in \cite{LEBMFtheory} that
\begin{equation}\label{ebvm}
\V(K, L) \geq \vol(K)^{\frac{n-p}{n}}\vol(L)^{\frac{p}{n}},
\end{equation}	 
with equality if and only if $K$ and $L$ are dilates of each other.  

The $L^p$ surface area measure of a function $f: \mathbb{R}^n \rightarrow \mathbb{R}$ with $L^p$ weak derivative is given by the lemma:

\begin{lemma}[Lemma 4.1 of \cite{LYZoptsobnorms}]\label{lem_functionaldualmixedvolume}
	Given $1 \leq p < \infty$ and a function $f: \mathbb{R}^n \rightarrow \mathbb{R}$ with $L^p$ weak derivative, there exists a unique finite Borel measure $S_p(f,.)$ on $\mathbb{S}^{n-1}$ such that
	\begin{equation}
		\label{def_Srf}
		\int_{\mathbb{R}^n}\phi(-\nabla f(x))^p dx = \int_{\mathbb{S}^{n-1}}\phi(u)^pdS_p(f,u),
	\end{equation}
for every non-negative continuous function $\phi : \mathbb{R}^n \rightarrow \mathbb{R}$ homogeneous of degree $1$.
\end{lemma} 
Conversely, for a \cb $L$ the function $f_L(x) = F(\|x\|_L)$ satisfies $S_p(f,.) = S_p(L,.)$ if $F$ is any function $F:\R_+ \to \R_+$ satisfying 
\[\int_0^\infty t^{n-1}F'(t)^p dt = 1\] (see \cite[Formula (4.13) ]{LYZoptsobnorms}). 
In view of identity \eqref{def_Srf},
for any $f$ and $L$ such that $S_p(f,.) = S_p(L,.)$, we have
\[\V(L, K) = \frac 1n \int_{\mathbb{R}^n} h(K, -\nabla f(x))^p dx.\]
This allows us to define the functional mixed volume. For $f$ a \cf[] we define
\begin{align*}
	\V(f, K) 
	&= \frac 1n \int_{S^{n-1}} h(K, \xi)^p dS_p(f, \xi)\\
	&= \frac 1n \int_{\mathbb{R}^n} h(K, -\nabla f(x))^p dx.\\
\end{align*}
The $L_p$ Sharp Sobolev inequality for general norms was proved by Cordero, Nazaret and Villani in \cite{cordero2004mass} using a mass-transportation method, and it extends the mixed volume inequality to the functional setting.
If $1\leq p < n$, $L$ is a centrally symmetric \cb and $f$ is a \cf[], we have
\begin{equation}
	\label{ineq_mixedvolume_f}
	\V(f,L) = \frac 1n \int_{\R^n} \|\nabla f(x)\|_{L^\circ}^p dx \geq \constcnv \|f\|_{p^*}^p \vol(L)^{p/n}
\end{equation}
where \eq $f(x) = a F_p(b \|x-x_0\|_L)$ with $x_0 \in \R^n$, $a,b > 0$ and
\[F_p(t) =\left\{ \begin{array}{cc} (1+t^{\frac p{p-1}} )^{1-\frac np} \hbox{ if } p \in (1,n)\\ \chi_{[0,1]}(t) \hbox{ if } p = 1.\end{array}\right. \]
For $p=1$ the equality is understood in the weak sense.
For $L=B_2$ we recover the euclidean $L_p$-Sobolev inequality
\[\constsob \| |\nabla f(x)|_2 \|_{L^p} \geq  \|f\|_{p^*}\]
with constant $\constsob = (n \constcnv \omega_n^{\frac pn})^{-1/p}$ where $\omega_n$ is the volume of the unit euclidean ball.
We refer to \cite[equation~(5)]{haddad2016sharp} for the precise value of $\constsob$.

\begin{definition}
	For a \cb $K$ and $1\leq p < n$ we define the function $\extrsobolev_{p,K}(x) = a F_p(\|x\|_K)$ where $a$ is defined by the requirement that
	\[\V(\extrsobolev_{p,K}, L) = \V(K,L)\]
	for every \cb $L$.
\end{definition}
Inserting $\extrsobolev_{p,L}$ in \eqref{ineq_mixedvolume_f} and using the equality case, we have 
\begin{equation}
	\constcnv \|\extrsobolev_{p,K}\|_{p^*}^p = \vol(K)^{\frac{n-p}n}.
\end{equation}

For $p\geq 1$, the $L_p$-dual mixed volume $\tV(K,L)$ of two bodies $K,L$ is defined by
\begin{equation}
	\label{def_dualmixedvolume_s}
	\tV(K,L) = \frac{n+p}{n} \int_K \|x\|_L^p dx
\end{equation}
and the $L_p$-dual mixed volume inequality states that
\[\tV(K,L) \geq \vol(K)^{\frac{n+p}n} \vol(L)^{-\frac pn}.\]

Definition \eqref{def_dualmixedvolume_s} extends easily to functions (see \cite[equation (2.6)]{lutwak2004moment}) as
\[\tV(f,L) = \frac{n+p}{n} \int_{\R^n} f(x) \|x\|_L^p dx.\]
For a \cb $K$, the function $f_K(x) = F(\|x\|_K)$ satisfies $\tV(f, L) = \tV(K, L)$ for every \cb $L$, provided that the function $F$ satisfies
\[(n+p) \int_0^\infty t^{n+p-1} F(t) dt = 1.\]
The next functional inequality is an extension to functions of the $L_p$-dual mixed volume inequality.
\begin{lemma}[{\cite[Lemma 4.1]{lutwak2004moment}}]
Let $\l \in \domlambda$ and $L$ be a \cb with the origin in its interior. Then, for any non-zero non-negative continuous function $f:\R^n \to \R$ with compact support,

\begin{equation}
	\label{ineq_moment}
	\tV(f, L) \geq \constmoment \|f\|_1^{\frac{n+p \l '}n} \|f\|_\l^{- \frac pn \l'} \vol(L)^{-\frac pn}.
\end{equation}
	We refer to \cite[Lemma 4.1]{lutwak2004moment} for the explicit value of the sharp constant $\constmoment$.
Moreover, \eq $f(x) = a \extrmoment_{p,\l}(b \|x\|_L)$ for constants $a, b \in \R$.
\end{lemma}

\begin{definition} For a \cb $K$ we define the function $\extrmoment_{p,\l, K}(x) = a \extrmoment_{p,\l}(\|x\|_K)$
where $a$ is defined by the requirement that $\tV(\extrmoment_{p,\l, K},L) = \tV(K, L)$ for every \cb $L$.
\end{definition}
Inserting $\extrmoment_{p,\l,L}$ in \eqref{ineq_mixedvolume_f} and using the equality case, we obtain
\[\constmoment \|\extrmoment_{p,\l,L}\|_1^{\frac{n+p \l '}n} \|\extrmoment_{p,\l,L}\|_\l^{- \frac pn \l'} = \vol(L)^{\frac{n+p}n}.\]

It is easy to see how the $L_p$-dual mixed volume relates to the $L_p$ Random Simplex Inequality, as
\[\tV(L_n, \N ^\circ(\mani L)) = \frac{n+p}n \I(\many L).\]
We also have $\I(\extrmoment_{p,\l, L_1}, \ldots, \extrmoment_{p,\l, L_n}) = \I(\many L)$ and $\tI(\extrsobolev_{p,L_1}, \ldots, \extrsobolev_{p,L_n}) = \tI(\many L)$.

For a \cb $K$ with smooth boundary and strictly \pc, the $L_p$ surface area measure of $K$ is absolutely continuous with respect to the invariant measure of the sphere and 
\[
	\int_{S^{n-1}} \phi(\xi) d S_{p,K}(\xi) = \int_{S^{n-1}} \phi(\xi) f_p(K,\xi) d \xi
\]
for every measurable function $\phi$, where $f_p$ is called the $L_p$ curvature function.
For $p=1$ this is the inverse of the Gauss curvature.
The $p$-affine surface area of $K$ is defined by
\begin{align*}
	\Omega_p(K) 
	&= \int_{S^{n-1}} f_p(K, \xi)^{\frac n{n+p}} d\xi\\
	&= \int_{S^{n-1}} f_p(K, \xi)^{-\frac p{n+p}} dS_{p,K}.\\
\end{align*}
For $p=1$ we obtain the affine surface area measure
\[\Omega(K) = \int_{\partial K} \kappa(x)^{\frac 1{n+1}} dx.\]

We refer to \cite[Section~10.5]{schneider2014convex} for more general definitions, and to \cite{meyer1997p} for a geometric interpretation.
$\Omega_p$ is a centro-affine invariant (invariant under volume preserving linear transformations), but not translation invariant unless $p=1$.
%
%
\section{Functional Random Simplex Inequality}
\label{sec_rsi}
Given functions $\many[k] l$ and convex bodies $L_{k+1}, \ldots, L_n$, we write $\I(\many[k] l, L_{k+1}, \ldots, L_n) = \I(\many l)$ where $l_i = \extrmoment_{p,\l,L_i}$ for $i = k+1, \ldots, n$. This quantity is independent of $\l$.

\begin{potwr}{thm_iso_f}
We prove by induction in $k$, the following 
\begin{claim*}
	Let $\many[k] l$ be \cf, and let $L_{k+1}, \ldots, L_{n-1}$ be convex bodies, then
	\[\vol(\Nc (\many[k] l, L_{k+1}, \ldots, L_{n-1})) \leq \constisos \prod_{i=1}^{k} \constmoment ^{-\frac np} \|l_i\|_1^{-\frac{n+p\lp}p} \|l_i\|_\l^{\lp} \prod_{j=k+1}^{n-1} \vol(L_j)^{- \frac{n+p}p}\]
	and equality holds if and only if $\many[k] l$ have the form $l_i(x) = a_i \extrsobolev_p(b_i |A.(x-x_i)|_2)$ where $x_i \in \R^n$, $A \in \gl$, and $L_i = a_i A^{-1}.B_2$ for $i \geq k+1$.
	Taking $k=n-1$ we obtain the Theorem with $\constisof = \constisos \constmoment^{-\frac {n^2}p}$.

\end{claim*}
	For $k=0$ the claim is exactly Theorem \ref{thm_iso_s}.
	Assume $k \geq 1$ and that the Lemma is true with $k$ replaced by $k-1$.
	Given functions $l_1, \ldots, l_k$ and sets $L_{k+1}, \ldots, L_{n-1}$ take 
	\[K = \N ^\circ (l_1, \ldots, l_{k-1}, \Nc (l_1, \ldots, l_{k},L_{k+1}, \ldots, L_{n-1}), L_{k+1}, \ldots, L_{n-1}).\]

	Using the commutativity of $\I$,
	\begin{align*}
		\tV(l_k, K)
		&= \frac{n+p}n \I(l_1, \ldots, l_{k-1}, \Nc (l_1, \ldots, l_{k},L_{k+1}, \ldots, L_{n-1}), L_{k+1}, \ldots, L_{n-1}, l_k)\\
		&= \frac{n+p}n \I(l_1, \ldots, l_{k-1}, l_k , L_{k+1}, \ldots, L_{n-1}, \Nc (l_1, \ldots, l_{k},L_{k+1}, \ldots, L_{n-1}))\\
		&= \tV(\Nc (l_1, \ldots, l_{k},L_{k+1}, \ldots, L_{n-1}), \Nc (l_1, \ldots, l_{k},L_{k+1}, \ldots, L_{n-1}))\\
		&= \vol(\Nc (l_1, \ldots, l_{k},L_{k+1}, \ldots, L_{n-1}))
	\end{align*}
	then by \eqref{ineq_moment} 
	\begin{align*}
		\vol(\Nc (l_1, \ldots,l_k, L_{k+1}, \ldots, L_{n-1}))
		&\geq \constmoment \|l_k\|_1^{\frac{n+p\lp}n} \|l_k\|_\l^{-\frac {p \lp}{n} } \vol(K)^{-\frac pn}.
	\end{align*}
	Notice that the definition of $K$ involves only $k-1$ functions.
	Then by the induction hypothesis we have
	\begin{align*}
		\vol(\Nc (l_1, &\ldots, l_k, L_{k+1}, \ldots, L_{n-1})) \geq \constmoment \|l_k\|_1^{\frac{n+p\lp}n} \|l_k\|_\l^{-\frac {p \lp}{n} }\\
		&\times \left( \constisos \prod_{i=1}^{k-1} \constmoment^{-\frac np} \|l_i\|_1^{-\frac{n+p\lp}p} \|l_i\|_\l^{\frac {p \lp}{p} }\vol(\Nc (l_1, \ldots, l_k, L_{k+1}, \ldots, L_{n-1}))^{- \frac{n+p}p} \prod_{j=k+1}^{n-1} \vol(L_j)^{- \frac{n+p}p} \right)^{-\frac pn}\\
		&= \constmoment \|l_k\|_1^{\frac{n+p\lp}n} \|l_k\|_\l^{-\frac {p \lp}{n} }\\
		&\times \left( \constisos \prod_{i=1}^{k-1} \constmoment^{-\frac np} \|l_i\|_1^{-\frac{n+p\lp}p} \|l_i\|_\l^{\lp} \prod_{j=k+1}^{n-1} \vol(L_j)^{- \frac{n+p}p} \right)^{-\frac pn} \vol(\Nc (l_1, \ldots, l_k, L_{k+1}, \ldots, L_{n-1}))^{\frac{n+p}n} \\
	\end{align*}
	that proves the Lemma.

	The equality case follows from the equality case of the induction hypothesis (the equality case of Theorem \ref{thm_iso_s} for $k=n-1$), and the equality case of \eqref{ineq_moment}.
\end{potwr}

We complete the proof of Corollary \ref{cor_rsi_f} by proving the equivalence between random and isoperimetric inequalities in the functional setting.

\begin{potwr}[Corollary]{cor_rsi_f}
	By the dual mixed volume inequality and Theorem \ref{thm_iso_f},
\begin{align*}
	\I(\many l) 
	&= \frac n{n+p} \tV(l_n, \N ^\circ(\mani l)) \\
	&\geq \frac n{n+p} \constmoment \|l_n\|_1^{\frac{n+p\lp}n} \|l_n\|_\l^{-\frac {p \lp}{n} } \vol(\N ^\circ(\mani l))^{-\frac pn}\\
	&\geq \constrsif \|l_n\|_1^{\frac{n+p\lp}n} \|l_n\|_\l^{-\frac {p \lp}{n} } \prod_{i=1}^{n-1} \|l_i\|_1^{\frac{n+p\lp}n} \|l_i\|_\l^{-\frac {p \lp}{n} }
\end{align*}
where $\constrsif = \frac n{n+p} \constmoment \constisof^{-\frac pn}$.
\end{potwr}

\section{The Dual Random Simplex Inequality}
\label{sec_rsid}
In this section we prove Theorems \ref{thm_rsid_s} and \ref{thm_rsid_f}.
Let $L$ be a \cb with smooth boundary and \pc, then it has a well defined $p$-curvature function $f_{p,L}$ defined as the density of the $p$-surface area measure with respect to the Lebesgue measure in the sphere, this is
\[d S_{p,L} = f_{p,L}(\xi) d \xi.\]

For $L$ as before, we consider the star body $L^*_p$ given by the radial function $r(L^*_p, \xi) = f_{p,L}(\xi)^{\frac 1{n+p}}$ and its volume given by
\begin{equation}
	\label{ident_paffinesurfacearea}
	\vol(L^*_p) = \frac 1n \int_{S^{n-1}}f_{p,L}(\xi)^{\frac n{n+p}} = \frac 1n \Omega_p(L).
\end{equation}

Now we prove Theorem \ref{thm_rsid_s}.

\begin{potwr}{thm_rsid_s}
	By definition of $\tI$ and $L^*_p$,
\begin{align}
	\label{proof_dualization_1}
	\tilde \I(L_1, \cdots, L_n)
	&= \int_{S^{n-1}} \cdots \int_{S^{n-1}} D_n(\xi_1, \ldots, \xi_n)^p dS_{p,L_1}(\xi_1) \ldots dS_{p,L_n}(\xi_n) \nonumber\\
	&= \int_{S^{n-1}} \cdots \int_{S^{n-1}} D_n(\xi_1, \ldots, \xi_n)^p \prod_{i=1}^n f_{p,L_i}(\xi_i) d \xi_1 \ldots d \xi_n \nonumber\\
	&= \int_{S^{n-1}} \cdots \int_{S^{n-1}} D_n(\xi_1, \ldots, \xi_n)^p \prod_{i=1}^n r_{{(L_i)}^*_p}(\xi_i)^{n+p} d \xi_1 \ldots d \xi_n \nonumber\\
	&= (n+p)^n \int_{(L_1)^*_p} \cdots \int_{(L_n)^*_p} D_n(x_1, \ldots, x_n)^p dx_1 \ldots dx_n\nonumber\\
	&= (n+p)^n \I({(L_1)}^*_p, \cdots, {(L_n)}^*_p)\nonumber\\
\end{align}
	The sets ${(L_i)}^*_p$ are not necessarily convex, but Corollary \ref{cor_rsi_f} applied to indicator functions with $\l = \infty$ implies that Theorem \ref{thm_rsi_s} remains valid for measurable sets (notice that $\constmoment[\infty] = 1$).
	By this inequality and identity \eqref{ident_paffinesurfacearea} we get
	\begin{align*}
	\tilde \I(L_1, \cdots, L_n) 
		& \geq (n+p)^n \constrsis \prod_{i=1}^n \vol((L_i)^*_p)^{\frac {n+p}n} = (n+p)^n \frac \constrsis {n^{n+p}} \prod_{i=1}^n \Omega_p(L_i)^{\frac {n+p}n}\\
	\end{align*}
\end{potwr}

The proof of Theorem \ref{thm_rsid_f} requires a simple lemma:
\begin{lemma}
	\label{lem_levelsets}
	Let $g:(0,\infty) \to [0,\infty)$ be a \cf[] and $\lambda \in \domlambda$, then
	\begin{equation}
		\label{ineq_levelsetfunction}
		\left( \int_0^\infty g(t)^{\frac{n+p}n} dt \right)^{\frac n{n+p}} 
		\geq \constlevelsets \left( \int_0^\infty g(t) t^{\lambda-1} dt \right)^{-\frac p{n+p} \frac 1{\l-1}}  \left( \int_0^\infty g(t) dt \right)^{\frac{n+p\lp}{n+p}}
	\end{equation}
	where 
	\[\constlevelsets = \left\{ \begin{array}{cc}
		(1-\lambda ) \left(\frac{p}{n+p}\right)^{\frac{p}{(\lambda -1) (n+p)}} \left(\lambda -\frac{n}{n+p}\right)^{\frac{n-\lambda  (n+p)}{(\lambda -1) (n+p)}}\frac{\Gamma \left(\frac{1}{1-\lambda }\right) }{\Gamma \left(\frac{n}{p}+2\right) \Gamma \left(\frac{\lambda }{1-\lambda }-\frac{n}{p}\right) }& \hbox{ if } \lambda < 1\\
		(\lambda -1) \left(\frac{p}{n+p}\right)^{\frac{p}{(\lambda -1) (n+p)}} \left(\lambda +\frac{p}{n+p}-1\right)^{\frac{-\lambda  n+n-\lambda  p}{(\lambda -1) (n+p)}} \frac{\Gamma \left(\frac{n}{p}+\frac{1}{\lambda -1}+2\right)}{\Gamma \left(\frac{\lambda }{\lambda -1}\right) \Gamma \left(\frac{n}{p}+2\right)} & \hbox{ if } \lambda > 1\\
		\end{array}
		\right.
	\]
	and \eq $g(t) = a p_\lambda(t/r)$ for some $a,r>0$ and 
	\[
		p_\lambda(t) = \left\{ 
			\begin{array}{cc}
			(t^{\l-1} - 1)_+^{\frac np} & \hbox{ if } \l<1\\
			(1-t^{\lambda-1})^{\frac{n}{p}}_+ & \hbox{ if } \l>1\\
			\chi_{[0,1]}(t) & \hbox{ if } \l=\infty.
			\end{array}
		\right.
	\]
	For $\lambda = \infty$ we interpret inequality \eqref{ineq_levelsetfunction} as
	\[\left( \int_0^\infty g(t)^{\frac{n+p}n} dt \right)^{\frac n{n+p}} \geq (\supp g)^{-\frac p{n+p} } \int_0^\infty g(t) dt\]
	where $\supp g$ is the minimum $a>0$ such that $g(t)=0$ for all $t>a$.
\end{lemma}
\begin{proof}
	For $\lambda > 1$ and $t>0$, let $p_{\lambda}(t) = (1-t^{\lambda-1})^{\frac{n}{p}}_+$.
	Then
		\[p_{\lambda}(t/r)^{\frac pn} \geq 1-t^{\lambda-1}r^{1-\lambda}.\]
	Multiplying by $g(t)$ and integrating
	\[
		\int_{0}^{\infty}g(t)p_{\lambda}(t/r)^{\frac{p}{n}}dt \geq \int_{0}^{\infty}g(t)dt - r^{1-\lambda}\int_{0}^{\infty}g(t)t^{\lambda-1}dt.
	\]
	By H\"older,
	\begin{align*}
		\int_{0}^{\infty}g(t) dt 
		&\leq \int_{0}^{\infty}g(t)p_{\lambda}(t/r)^{\frac{p}{n}}dt + r^{1-\lambda}\int_{0}^{\infty}g(t)t^{\lambda-1}dt\\
		&\leq  \left(\int_{0}^{\infty} g(t)^{\frac{n+p}{n}}dt\right)^{\frac{n}{n+p}}\left(\int_{0}^{\infty}p_{\lambda}(t/r)^{\frac{n+p}{n}}dt\right)^{\frac{p}{n+p}} + r^{1-\lambda}\int_{0}^{\infty}g(t)t^{\lambda-1}dt\\
		&\leq  \left(\int_{0}^{\infty} g(t)^{\frac{n+p}{n}}dt\right)^{\frac{n}{n+p}}\left(\int_{0}^{\infty}p_{\lambda}(t)^{\frac{n+p}{n}}dt\right)^{\frac{p}{n+p}}r^{\frac{p}{n+p}} + r^{1-\lambda}\int_{0}^{\infty}g(t)t^{\lambda-1}dt.
	\end{align*}
	Minimizing the right-hand side for $r>0$, we obtain the result for $\lambda > 1$.
	For the case $\lambda \in \left(\frac{n}{n+p}, 1\right)$,
	we define $q_\l(t) = (t^{\lambda-1} - 1)_+^{\frac{n}{p}}$.
	Then, $q_{\lambda}(t)^{\frac{p}{n}} \geq t^{\lambda-1} - 1$
	and it follows that
	\[
		\int_{0}^{\infty}g(t)q_{\lambda}\left(t/r\right)^{\frac{p}{n}}dt \geq r^{1 - \lambda}\int_{0}^{\infty}g(t)t^{\lambda - 1}dt - \int_{0}^{\infty}g(t)dt.
	\]
	By H\"older
	\begin{align*}
		\int_{0}^{\infty}g(t)t^{\lambda - 1}dt
		&\leq r^{\lambda-1}\int_{0}^{\infty}g(t)dt + r^{\lambda-1}\left(\int_{0}^{\infty} g(t)^{\frac{n+p}{n}}dt\right)^{\frac{n}{n+p}}\left(\int_{0}^{\infty}q_{\lambda}\left(t/r\right)^{\frac{n+p}{n}}dt\right)^{\frac{p}{n+p}}\\
		&\leq r^{\lambda-1}\int_{0}^{\infty}g(t)dt + r^{\lambda-\frac{n}{n+p}}\left(\int_{0}^{\infty} g(t)^{\frac{n+p}{n}}dt\right)^{\frac{n}{n+p}}\left(\int_{0}^{\infty}q_{\lambda}(t)^{\frac{n+p}{n}}dt\right)^{\frac{p}{n+p}}
	\end{align*}
	and again we conclude minimizing with respect to $r>0$.
	The case $\lambda = \infty$ is just an application of H\"older inequality.
\end{proof}  

\begin{potwr}{thm_rsid_f}
Let $\many l$ be \cfs.
By Sard's Lemma, almost every level set of each function $l_i$ is a $C^2$-smooth manifold.
For such a function $l=l_i$, the relation \eqref{def_Srf} implies that the surface area measure of $l$ is absolutely continuous with respect to the invariant measure in $S^{n-1}$, then we also have a well defined $p$-curvature function $f_{p,l}$ such that
$d S_{p,l} = f_{p,l}(\xi) d \xi$.
We define the star body $(l)^*_p$ similarly, by
\[r((l)^*_p, \xi)^{n+p} = f_{p,l}(\xi).\]
As in \eqref{proof_dualization_1} we compute
\begin{align}
	\label{functionalineq}
	\tI(l_1, \cdots, l_n) = \I( (l_1)^*_p, \ldots, (l_n)^*_p) \geq \constrsis \prod_{i=1}^n \vol((l_i)^*_p)^{\frac{n+p}{n}}.
\end{align}

Fix $l = l_i$ for $i = 1, \ldots n$.
Consider the function $n^l:\R^n \to S^{n-1}$ given by $n^l(x) = \frac{\nabla l(x)}{|\nabla l(x)|_2}$.
Again by Sard's Lemma applied to $n^l$, for almost every $\xi$, the set
\[P(\xi) = \{x \in \R^n\ /\ \nabla l(x) \neq 0 \hbox{ and } n^l(x) = \xi\} \]
is a $1$-dimensional submanifold of the open subset of $\R^n$ where $\nabla l \neq 0$.
By the definition of $f_{p,l}$ we have for any $p$-homogeneous function $\phi$ on $\R^n$,
\begin{align}
	\label{proof_compfpl}
	\int_{S^{n-1}} \phi(\xi) d S_{p, l} 
	&= \int_{\R^n} \phi(\nabla l(x)) dx\nonumber\\
	&= \int_{S^{n-1}} \int_{P_l(\xi)} J_{n^l}(x)^{-1} \phi(\nabla l(x)) dx \ d\xi\nonumber\\
	&= \int_{S^{n-1}} \int_{P_l(\xi)} J_{n^l}(x)^{-1} |\nabla l(x)|^p dx \ \phi(\xi) d\xi 
\end{align}
where we used the generalised co-area formula \cite{nicolaescu2011coarea}, and $J_{n^l}$ is the Jacobian of the function $n^l:\R^n \to S^{n-1}$ given by $n^l(x) = \frac{\nabla l(x)}{|\nabla l(x)|_2}$.
The Jacobian of a transformation at a point $x_0$ can be computed by the formula (see \cite[Lemma~1.2]{nicolaescu2011coarea})
\[J = \frac{ D_{n-1}(A.u_1, \ldots, A.u_{n-1}) |u_n|}{ D_{n}(u_1, \ldots, u_{n-1}, u_n)}\]
where $A$ is the differential of $n^l$ at $x_0$, $u_n$ is a vector generating the kernel of $A$ and $u_1, \ldots, u_{n-1}$ is any basis of a complementary subspace to $\ker(A)$.
Now, taking $u_1, \ldots, u_{n-1}$ to be an orthonormal basis of the orthogonal space to $\xi$, and $u_n$ unitary, we obtain
$J = | \kappa / \langle \xi, u_n \rangle |$
where $\kappa$ is the Gauss curvature of the level set of $l$ at $x_0$.
Observe that $u_n$ is the unit tangent vector of the curve $P_l(\xi)$ at $x_0$.
From \eqref{proof_compfpl} we obtain
\begin{equation}
	\label{proof_form_f_kappa}
	f_{p,l}(\xi) = \int_{P_l(\xi)} \kappa(x)^{-1}  |\nabla l(x)|^{p} |\xi^T| d \mathcal H_1(x)
\end{equation}
where $|\xi^T|$ is the tangential component over the curve $P_l(\xi)$.

Since $l$ has convex level sets with \pc, each $P(\xi)$ can be parametrized by a curve $\gamma_\xi:(0, \|l\|_\infty) \to \R$ satisfying $l(\gamma_\xi(t)) = t$. We compute
\begin{align*}
	f_{p,l}(\xi) 
	&= \int_0^{\|l\|_\infty} \kappa(\gamma_\xi(t))^{-1} |\nabla l(\gamma_\xi(t))|^{p} |\langle \xi, \gamma_\xi'(t) \rangle| dt\\
	&= \int_0^{\|l\|_\infty} \kappa(\gamma_\xi(t))^{-1} |\nabla l(\gamma_\xi(t))|^{p-1} |\langle \nabla l(\gamma_\xi(t)), \gamma_\xi'(t) \rangle| dt\\
	&= \int_0^{\|l\|_\infty} \kappa(\gamma_\xi(t))^{-1} |\nabla l(\gamma_\xi(t))|^{p-1}  dt.\\
\end{align*}
By formula \eqref{proof_form_f_kappa}, applying the parametrization given by $\gamma_\xi$, the Minkowski integral inequality and the definition of the surface area measure we have
\begin{align*}
	\vol(l^*_p)
	&= \int_{S^{n-1}} \left( \int_0^{\|l\|_\infty} \kappa(\gamma_\xi(t))^{-1} |\nabla l(\gamma_\xi(t))|^{p-1} dt \right)^{\frac{n}{n+p}} d\xi\\
	&\geq \left( \int_0^{\|l\|_\infty} \left( \int_{S^{n-1}} \kappa(\gamma_\xi(t))^{-\frac n{n+p}} |\nabla l(\gamma_\xi(t))|^{\frac {(p-1)n}{n+p}} d\xi \right)^{\frac{n+p}{n}} dt\right) ^{\frac{n}{n+p}} \\
	&= \left( \int_0^{\|l\|_\infty} \left( \int_{\partial N_{l,t}} \kappa(x)^{\frac p{n+p}} |\nabla l(x)|^{\frac {(p-1)n}{n+p}} d S_{N_{l,t}}(x) \right)^{\frac{n+p}{n}} dt\right) ^{\frac{n}{n+p}}.\\
\end{align*}
	The Gauss curvature of the level set of a function can be computed taking the differential of $n^l$ restricted to $\xi^\bot$ giving
	$\kappa = \det \left( \begin{array}{c|c} 0 & \xi^T \\ \hline \\ \xi & |\nabla l|^{-1} Hl \end{array} \right)$,
	so we have $|\kappa| = |\det(K l)| |\nabla l|^{-n-1}$  and
	\[
		\vol(l^*_p) =\left( \int_0^{\|l\|_\infty} \left( \int_{\partial N_{l,t}} |\det(K l(x))|^{\frac p{n+p}} |\nabla l(x)|^{-1} d S_{N_{l,t}}(x) \right)^{\frac{n+p}{n}} dt\right) ^{\frac{n}{n+p}}.
	\]
	Let us define $\Omega_p(l,t) = \int_{\partial N_{l,t}} |\det(K l(x))|^{\frac p{n+p}} |\nabla l(x)|^{-1} d S_{N_{l,t}}(x)$, then a simple computation shows that
	\[\Omega_p(l^\alpha, t^\alpha) = \Omega_p(l,t) (\alpha t^{\alpha-1})^{(p-1)\frac n{n+p}},\]
	and taking $\lambda = 1+(\alpha-1) (n+1) \frac p{n+p}$ and $g(t) = \Omega_p(l,t)$ in Lemma \ref{lem_levelsets},
\begin{align*}
	\vol(l^*_p)
	&\geq \left( \int_0^\infty \Omega_p(l,t)^{\frac{n+p}{n}} dt\right)^{\frac{n}{n+p}}\\
	&\geq \constlevelsets \left( \int_0^\infty \Omega_p(l,t) dt \right)^{\frac{n+p\lp}{n+p}} \left( \int_0^\infty \Omega_p(l,t) t^{\lambda-1} dt \right)^{-\frac p{(n+p)(\l-1)}}\\
	&= \alpha^{\frac p{\l-1}}\constlevelsets \left( \int_0^\infty \Omega_p(l,t) dt \right)^{\frac{n+p\lp}{n+p}} \left( \int_0^\infty \Omega_p(l,t) (\alpha t^{\alpha-1})^{(p-1)\frac n{n+p}} (\alpha t^{\alpha-1}) dt \right)^{-\frac p{(n+p)(\l-1)}}\\
	&= \alpha^{\frac p{\l-1}} \constlevelsets \left( \int_0^\infty \Omega_p(l,t) dt \right)^{\frac{n+\alpha'}{n+1}} \left( \int_0^\infty \Omega_p(l^\alpha,s) ds \right)^{-\frac 1{n+1} \frac 1{\alpha-1}}\\
	&= \alpha^{\frac p{\l-1}} \constlevelsets \Omega_p(l)^{\frac{n+\alpha'}{n+1}} \Omega_p(l^\alpha)^{-\frac 1{n+1} \frac 1{\alpha-1}}
\end{align*}
and we obtain the result with $\constrsidf = \constrsis(\frac{\alpha^{\frac p{\l-1}}}n \constlevelsets)^{n+p}$.


	The equality case follows from the equality case of Theorem \ref{thm_rsi_s} in \eqref{functionalineq}, the equality case of Lemma \ref{lem_levelsets} for the function $\Omega_p(l,t)$, and the formula
	$\Omega_p(F(|x|_2), F(t)) = n \omega_n F'(t)^{(p-1) \frac n{n+p}}$.
\end{potwr}

\section{Open problems}
\label{sec_open}
In this section we discuss the dual inequality \eqref{conjineq_rsid_s_first} and its relation to the body $\tN$ and the Petty conjecture. 
Let us consider the following open problem:
\begin{problem}
	\label{conj_baseconjecture}
	Let $L$ be a \cb and $1\leq p <n$, then 
	\begin{equation}
		\label{conjineq_baseconjecture}
		\tI(\many L) \geq \constconjrsids \prod_{i=1}^n \vol(L_i)^{\frac{n-p}n}
	\end{equation}
	and \equas{\many L}.
\end{problem}

Now we show that the proof of Theorem \ref{thm_iso_f} works for $\tI$ and $\tN$ using the mixed volume.
As explained in the introduction, inequality \eqref{conjineq_rsid_s_first} is equivalent to the isoperimetric inequality $\vol(\tN(\mani L)) \geq \constconjisods \prod_{i=1}^{n-1} \vol(L_i)^{\frac{n-p}p}$. The proof of this equivalence is similar.
\begin{lemma}
	\label{lem_isod_sf}
	Assume Problem \ref{conj_baseconjecture} holds. Let $\mani l$ be \cf, let $L_{k+1}, \ldots, L_{n-1}$ be convex bodies and $1\leq p<n$, then
	\[\vol(\tN(\many[k] l, L_{k+1}, \ldots, L_{n-1})) \geq \constconjisods \prod_{i=1}^k \constcnv^{\frac np} \|l_i\|_{p^*}^n \prod_{j=k+1}^{n-1} \vol(L_j)^{\frac{n-p}p}\]
	and equality holds if and only if $\many[k] l$ have the form $l_i(x) = a_i \extrsobolev_p(b_i |A.(x-x_i)|_2)$ where $x_0 \in \R^n$, $A \in \gl$, and $L_i = a_i A^{-1}.B_2$ for $i \geq k+1$.
\end{lemma}
\begin{proof}
Take $K = \tN(\many[k-1] l, \tN(\many[k-1] l, l_k, L_{k+1}, \ldots, L_{n-1}), L_{k+1}, \ldots, L_{n-1})$.
	As in the proof of Lemma \ref{thm_iso_f},
	\[
	\V(l_k,K) = \vol(\tN(\many[k] l, L_{k+1}, \ldots, L_{n-1}))
	\]
then by the functional mixed volume inequality \eqref{ineq_mixedvolume_f} and the induction hypothesis we get
\begin{align*}
	\vol(\tN(&\many[k] l, L_{k+1}, \ldots, L_{n-1})) \geq \constcnv \vol(K)^{\frac pn} \|l_k\|_{p^*}^p\\
	&\geq \constcnv \left( \constconjisods \prod_{i=1}^{k-1} \constcnv^{\frac np}\|l_i\|_{p^*}^n \vol(\N (\many[k] l, L_{k+1}, \ldots, L_{n-1}))^{\frac{n-p}p} \prod_{j=k+1}^{n-1} \vol(L_j)^{\frac{n-p}p} \right)^{\frac pn} \|l_k\|_{p^*}^p\\
\end{align*}
and that proves the lemma.
	The equality case follows from the equality case of the induction hypothesis (the equality case of Problem \ref{conj_baseconjecture} for $k=n-1$), and the equality case of \eqref{ineq_mixedvolume_f}.
\end{proof}

Summarizing, Problem \ref{conj_baseconjecture} is equivalent to the following sharp inequalities
	\begin{equation}
		\label{conjineq_rsid_f}
		\tI(\many l) \geq \constconjrsidf \prod_{i=1}^n \|l_i\|_{p^*}^{p}
	\end{equation}
	\begin{equation}
		\label{conjineq_isod_s}
		\vol(\tN(\mani L)) \geq \constconjisods \prod_{i=1}^{n-1} \vol(L_i)^{\frac{n-p}p}
	\end{equation}
	\begin{equation}
		\label{conjineq_isod_f}
		\vol(\tN(\mani l)) \geq \constconjisodf \prod_{i=1}^n \|l_i\|_{p^*}^{n}
	\end{equation}
	and \equas{\many L} and $\many l$ have the form $l_i(x) =a_i F_p(b_i |A.(x-x_i)| )$ with $x_0 \in \R^n$, $a_i, b_i > 0$, $A \in \gl$.
	Here the constants satisfy $\constconjrsids = n \constconjisods^{p/n}$, $\constconjrsidf = n \constcnv \constconjisodf^{p/n}$, $\constconjisodf = \constconjisods \constcnv^{(n-1)\frac np}$.
Inequalities \eqref{conjineq_baseconjecture} and \eqref{conjineq_isod_s} are equivalent, as well as inequalities \eqref{conjineq_rsid_f} and \eqref{conjineq_isod_f}.
Also by Lemma \ref{lem_isod_sf}, \eqref{conjineq_isod_s} implies \eqref{conjineq_isod_f}, while the converse is obvious.

As a motivation to these problems we make two remarks, first notice that a particular case of inequality \eqref{conjineq_rsid_f} is a sharp, affine invariant, $L_p$ Sobolev-like inequality
\begin{equation}
	\label{ineq_sobolevish}
	\left( \manyint{\R^n} D_n(\nabla f(x_1), \ldots, \nabla f(x_n))^p \manydx \right)^{\frac 1{np}} \geq \constconjrsidf^{\frac 1{np}} \|f\|_{p^*}.  
\end{equation}

Second, for $p=1$, we have $\frac 1{n!} \tI[1](\man L) = \vol(\Pi L)$,
and inequality \eqref{conjineq_baseconjecture} for $L_1 = \cdots = L_n = L$ becomes Petty's conjectured inequality \eqref{ineq_pettyconj}.
By the Blaschke-Santal\'o inequality for symmetric bodies \cite[(10.28)]{schneider2014convex}, inequality \eqref{ineq_pettyconj} is stronger than the Petty Projection inequality 
\begin{equation}
	\label{ineq_pettyproj}
	\vol(\Pi^\circ L)^{-\frac 1n} \geq \frac{\omega_{n-1}}{\omega_n} \vol(L)^{\frac {n-1}n}.
\end{equation}
Zhang proved in \cite{zhang1999affine} a functional version of \eqref{ineq_pettyproj} that results in a Sharp Affine Sobolev inequality
\begin{equation}
	\label{ineq_zhang_1}
	\left( \frac 1n \int_{S^{n-1}} \left( \frac 12 \int_{\R^n} |\langle \nabla f(x), \xi \rangle| dx \right)^{-n} d\xi \right)^{-\frac 1n}
	\geq
	\frac{\omega_{n-1}}{\omega_n} \|f\|_{\frac n{n-1}}
\end{equation}
where the left-hand side is $\vol(\Pi^\circ K_f)^{-\frac 1n}$ and $K_f$ is such that $S(K_f,\cdot) = S(f,\cdot)$. Again by the Blaschke-Santal\'o inequality applied to $\Pi K_f$ we obtain
\begin{equation}
	\label{ineq_stronger_1}
	\left( \frac 1n \int_{S^{n-1}} \left( \frac 12 \int_{\R^n} |\langle \nabla f(x), \xi \rangle| dx \right)^{-n} d\xi \right)^{-\frac 1n}
	\geq
	\left( \frac 1{\omega_n^2 n!} \manyint{\R^n} D_n(\nabla f(x_1), \ldots, \nabla f(x_n)) \manydx \right)^{\frac 1n},
\end{equation}
meaning that the Sobolev-like inequality \eqref{ineq_sobolevish} for $p=1$ is stronger and directly implies the affine Sobolev inequality of Zhang \eqref{ineq_zhang_1}. One can check that the value of $\constconjrsidf[1]^{1/n}$ in \eqref{ineq_sobolevish} is consistent with inequality \eqref{ineq_zhang_1}.
Notice that any purely analytic proof of \eqref{ineq_stronger_1} must contour (or re-prove) the Blaschke-Santal\'o inequality. In any case, that doesn't seem a trivial thing to do, but it might extend to the $L_p$ case:
\begin{problem}
\begin{equation}
	\label{ineq_stronger_p}
	\left( \int_{S^{n-1}} \left( \int_{\R^n} |\langle \nabla f(x), \xi \rangle|^p dx \right)^{-\frac np} d\xi \right)^{-\frac 1n}
	\geq
	\constzhang \left( \manyint{\R^n} D_n(\nabla f(x_1), \ldots, \nabla f(x_n))^p \manydx \right)^{\frac 1{np}}
\end{equation}
\end{problem}

Lastly, to this moment we don't know whether Problem \ref{conj_baseconjecture} for $p>1$ can be reduced to the particular case where $L_1 = \cdots = L_n = L$.
We believe these problems are worth to consider since they might shed some light into an important open problem that resisted the symmetrization approach.

\bibliographystyle{plain}
\bibliography{ref}
\end{document}